\documentclass[runningheads]{llncs}

\usepackage{etex} 

\usepackage{url}
\usepackage{amsmath,amsfonts,amssymb}
\usepackage{verbatim}
\usepackage{stmaryrd}
\usepackage[inline]{enumitem}
\usepackage[ruled,vlined,linesnumbered]{algorithm2e}
\usepackage[font=scriptsize,labelfont=bf]{caption}
\usepackage[all]{xy}
\usepackage{overpic}
\usepackage{epic,eepic}
\usepackage{booktabs}
\usepackage[usenames,dvipsnames]{color}

\usepackage{array}
\usepackage{multirow}
\usepackage{color}
\usepackage{bm}
\usepackage{marvosym}
\usepackage{url}

\usepackage{graphicx}
\usepackage{wrapfig}
\usepackage{subfig}

\usepackage[firstpage]{draftwatermark} 
\SetWatermarkText{\hspace*{4.4in}\raisebox{5.13in}{\includegraphics[scale=0.1]{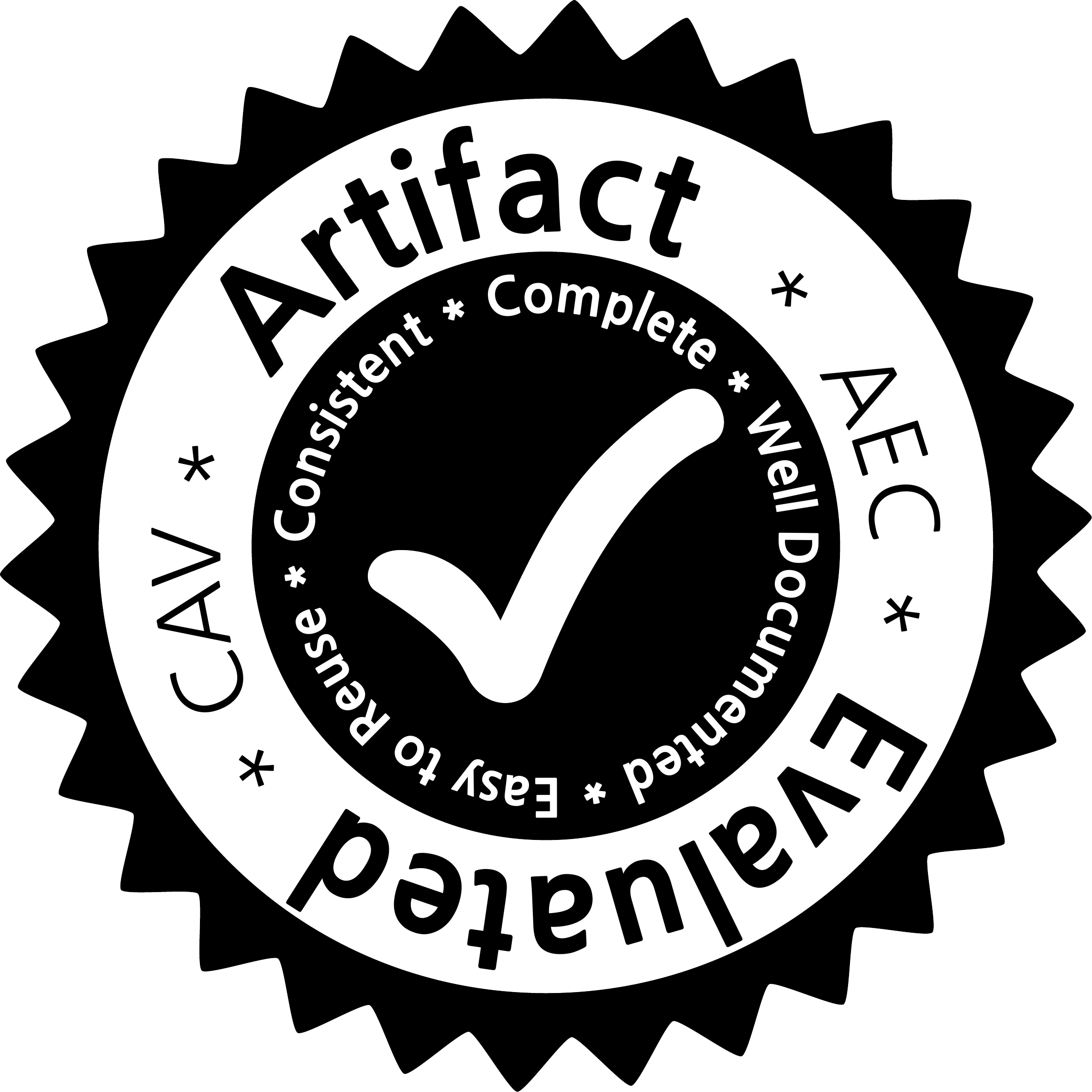}}}
\SetWatermarkAngle{0}

\usepackage{times}

\usepackage[title,page]{appendix}

\usepackage{tikz}
\tikzset{mynode/.style={draw,solid,circle,inner sep=1pt}}

\usepackage{scalerel} 

\usepackage{etoolbox} 
\patchcmd{\paragraph}{\itshape}{\bfseries\boldmath}{}{} 

\makeatletter
\RequirePackage[bookmarks,unicode,colorlinks=true]{hyperref}%
\def\@citecolor{blue}%
\def\@urlcolor{blue}%
\def\@linkcolor{blue}%

\def\orcidID#1{\smash{\href{http://orcid.org/#1}{\protect\raisebox{-1.25pt}{\protect\includegraphics{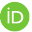}}}}}
\makeatother

\include{def}

\graphicspath{{figures/}}

\begin{document}

\title{Unbounded-Time Safety Verification of Stochastic Differential Dynamics\thanks{This work was partially funded by NSFC under grant No.\ 61625206, 61732001 and 61872341, by the ERC Advanced Project FRAPPANT under grant No.~787914, by the US NSF under grant No.~CCF 1815983 and by the CAS Pioneer Hundred Talents Program under grant No. Y8YC235015.}
}
\titlerunning{Unbounded-Time Safety Verification of Stochastic Differential Dynamics}


\author{
Shenghua Feng\inst{1,2}$^{\text{(\Letter)}}$\orcidID{0000-0002-5352-4954}
\and
Mingshuai Chen\inst{3}$^{\text{(\Letter)}}$\orcidID{0000-0001-9663-7441}
\and
Bai Xue\inst{1,2}$^{\text{(\Letter)}}$\orcidID{0000-0001-9717-846X} 
\and
Sriram Sankaranarayanan\inst{4}$^{\text{(\Letter)}}$\orcidID{0000-0001-7315-4340}
\and
Naijun Zhan\inst{1,2}$^{\text{(\Letter)}}$\orcidID{0000-0003-3298-3817}
}
\authorrunning{S. Feng et al.}
\institute{
SKLCS, Institute of Software, CAS, Beijing, China
\and
University of Chinese Academy of Sciences, Beijing, China\\
\email{\{fengsh,xuebai,znj\}@ios.ac.cn}
\and
Lehrstuhl f{\"u}r Informatik 2, RWTH Aachen University, Aachen, Germany\\
\email{chenms@cs.rwth-aachen.de}
\and
University of Colorado, Boulder, USA\\
\email{sriram.sankaranarayanan@colorado.edu}
}

\maketitle

\setcounter{footnote}{0}


\begin{abstract}

In this paper, we propose a method for bounding the probability that a stochastic differential equation (SDE) system violates a safety specification over the infinite time horizon. SDEs are mathematical models of stochastic processes that capture how states evolve continuously in time. They are widely used in numerous applications such as engineered systems (e.g., modeling how pedestrians move in an intersection), computational finance (e.g., modeling stock option prices), and ecological processes (e.g., population change over time). Previously the safety verification problem has been tackled over finite and infinite time horizons using a diverse set of approaches. The approach in this paper attempts to connect the two views by first identifying a finite time bound, beyond which the probability of a safety violation can be bounded by a negligibly small number. This is achieved by discovering an exponential barrier certificate that proves exponentially converging bounds on the probability of safety violations over time. Once the finite time interval is found, a finite-time verification approach is used to bound the probability of violation over this interval. We demonstrate our approach over a collection of interesting examples from the literature, wherein our approach can be used to find tight bounds on the violation probability of safety properties over the infinite time horizon.

\end{abstract}

\section{Introduction}\label{sec:intro}

In this paper, we investigate the problem of verifying probabilistic safety
properties for continuous stochastic dynamics modeled by stochastic differential equations (SDEs). The study of SDEs dates back to the 1900s when, e.g., Einstein used SDEs to model the phenomenon of Brownian motion~\cite{Einstein/1906/Theory}.
Since then, SDEs have witnessed numerous applications including models of disturbances in engineered systems ranging from wind forces~\cite{Wang+Others/2015/Long} to pedestrian motion~\cite{Hoogendoorn+Others/2004/Pedestrian}; models of financial instruments such as options~\cite{Black+Scholes/1973/Pricing}; and models of biological/ecological processes for instance predator-prey models~\cite{panik2017stochastic}. In the meantime, SDEs are hard to reason about:  they are defined using ideas from stochastic calculus that reimagine basic concepts such as integration in order to conform to the basic laws of probability and stochastic processes~\cite{oksendal2013stochastic}.

There are many important verification problems for SDEs. Prominent topics
include the safety verification problem which seeks to know the probability that
a given SDE with specified initial conditions will enter
an unsafe region (or leave a safe region) over a given
time horizon. Generally, safety verification can be performed
over a finite-time horizon setting, wherein the probability
is sought over a finite time interval $[0, T]$. On the other hand,
the infinite-time horizon problem seeks a bound on the probability
of satisfying a safety property over the unbounded time horizon
$[0, \infty)$. A handful of methods have been proposed for verifying SDE systems, such as the barrier certificate-based methods over both the infinite time horizon~\cite{prajna2004stochastic} and finite time horizons~\cite{steinhardt2012finite}, the moment optimization-based method over
finite time horizons~\cite{sloth2015safety} and the Hamilton-Jacobi-based
method over the infinite time horizon~\cite{Koutsoukos2008}. 
The novelty of our work lies in the reduction of infinite-time
horizon verification problems to finite time problems.
%

In this paper, we propose a novel reduction-based method to verify unbounded-time safety properties of stochastic systems modeled as nonlinear polynomial SDEs. We employ a similar idea as in~\cite{feng2019taming} (for verifying delay differential equations) that reduces the safety verification problem over the infinite time horizon to the one over a finite time interval. This is achieved by computing an \emph{exponential stochastic barrier certificate} which witnesses an exponentially decreasing upper bound on the probability that a target system violates a given safety specification. Consequently, for any $\epsilon >0$, we can identify a time instant $T$ beyond which the violation (a.k.a. failure) probability is smaller than the negligibly small cutoff $\epsilon$. The reduced bounded-time safety verification problem over $[0, T]$ can hence be tackled by any of the available methods. We furthermore present an alternative method to address the reduced finite-time horizon verification problem based on the discovery of a \emph{time-dependent stochastic barrier certificate}. 
We show that both the exponential and the time-dependent stochastic barrier certificate can be synthesized by respectively solving a pertinent \emph{semidefinite programming} (SDP)~\cite{wolkowicz2012handbook} optimization problem. Experimental results on some interesting examples taken from the literature demonstrated the effectiveness of the reduction and that our method often produces tighter bounds on the failure probability.
Our approach has some broad similarities to related approaches in symbolic execution of probabilistic programs that conclude facts about infinitely many behaviors by analyzing finitely many paths in the program that account for a sufficient probability among all the behaviors~\cite{DBLP:conf/pldi/SankaranarayananCG13}.
%

\paragraph*{Contributions.}
The main contributions of this work can be summarized as follows: (1) We reduce
the unbounded-time safety verification of stochastic systems to a bounded
one, based on an exponentially decreasing bound on the failure probability which guarantees the dominance of the overall failure probability by the truncated finite time horizon. (2) We show how the obtained bound on the overall failure
probability is tighter than that produced by existing methods for some interesting SDEs.

\paragraph*{Related Work.}
The use of mathematical models of processes --ranging from
finite state machines to various types of differential equations--
has allowed us to reason about rich behaviors of Cyber-Physical
Systems produced by the interaction between digital computers and physical plants~\cite{rajkumar2010cyber}. In this regard,
many modeling formalisms have been studied including
finite state machines, ordinary differential equations (ODEs),
timed automata, hybrid automata, etc.~\cite{Deshmukh_Sankaranarayanan__2019__Formal}, on top of which a large variety of verification problems have been extensively investigated, e.g., safety verification through reachability 
analysis and temporal logic verification~\cite{Baier+Katoen/2008/Principles}.

In the existing literature on formal verification, ODEs are often used to describe the behavior of
deterministic continuous-time systems. However, these models have been shown over-simplistic
in many applications that involve time delays, nondeterministic inputs and stochastic
noises. 
SDEs hence arose as an important class of models that have been employed in practical
domains covering, among others~\cite{oksendal2013stochastic}, financial models such as the famous
Black-Scholes model used extensively in the theory of options
pricing~\cite{Black+Scholes/1973/Pricing}, wind
disturbances~\cite{Wang+Others/2015/Long}, human pedestrian
motion~\cite{Hoogendoorn+Others/2004/Pedestrian} and ecological
models~\cite{panik2017stochastic}.

In what follows, we place our work in the context of formal verification techniques tailored for stochastic differential dynamics modeled as SDEs, and discuss contributions thereof that are highly related to our approach. Unbounded-time stochastic safety verification of SDE systems was first studied by Prajna et al. in~\cite{prajna2004stochastic,DBLP:journals/tac/PrajnaJP07}, where a typical supermartingale was employed as a stochastic barrier certificate followed by
computational conditions derived from Doob's martingale
inequality~\cite{karatzas2014brownian}. Thereafter, the stochastic
barrier certificate-based method was extended to cater for bounded-time safety
verification by Steinhardt and Tedrake~\cite{steinhardt2012finite} by leveraging a relaxed formulation called $c$-martingale for locally stable systems. The barrier
certificate-based method by Prajna et al. (ibid.) for unbounded-time
safety verification often leads to conservative bound on the failure probability. On the other hand,
Steinhardt and Tedrake (ibid.) established impressive probability bounds but only for
finite time horizons. In order to
reduce the conservativeness, we propose a method of reducing the unbounded
safety verification to a bounded one. Although our method in this
paper is also based on the construction of stochastic barrier
certificates, the gain of stochastic barrier certificates only helps
to identify a finite time interval such that the violation
probability of interest beyond this time interval is arbitrarily negligibly small. A
time-dependent barrier certificate is further proposed to solve the
resulting bounded-time safety verification. The Unbounded-time safety verification problem
has also been studied by Koutsoukos and Riley~\cite{Koutsoukos2008}, who linked the
reachability probability to the viscosity solution of certain
Hamilton-Jacobi partial differential equations, under restrictions on bounded state space and non-degenerate diffusion. Grid-based numerical
approaches, e.g., the finite difference method in~\cite{Koutsoukos2008} and the level set method in~\cite{MT05}, are traditionally used to solve these equations, leading to
the fact that the Hamilton-Jacobi reachability method only scales well
to systems of special structures. More recently, a novel constraint solving-based method has been proposed in~\cite{liureachability} for algebraically over- and under-approximating the reachability probability, which is nevertheless limited to bounded-time safety verification. In addition to the abovementioned methods, we refer the readers to~\cite{Bujorianu04} for a Dirichlet form-based method for stochastic hybrid systems featuring ``nice'' Markov properties, while to~\cite{Younes+Simmons/2002/Probabilistic,BBK07,Lygeros10} and~\cite{APLS08,KD01} respectively for related contributions in statistical and discrete/numerical methods for stochastic verification and control.

	Finally, we mention a relation between the ideas in this paper and previously proposed ideas for (non-stochastic) ODEs due to Sogokon et al.~\cite{Sogokon+Others/2018/Vector}. The key similarity lies in the use of a non-negative matrix through which a vector of functions whose derivatives are related to their current value. Whereas Sogokon et al. explored this idea for ODEs, we do so for SDEs. Another significant difference, in our work, is that we use the super-martingale functions to identify a time horizon $[0, T]$ and bound the probability of safety violation beyond $T$.

The reminder of this paper is structured as follows. Sect.~\ref{sec:preliminaries} introduces stochastic differential
dynamics modeled by SDEs and the unbounded-time safety verification problem of
interest. Sect.~\ref{sec:reduction} elucidates the reduction of unbounded safety verification to bounded ones based on the witness of stochastic barrier certificates. Sect.~\ref{sec:synthesis} presents the SDP
formulation for discovering such barrier certificates over the reduced  bounded time interval. After demonstrating our method on several examples in Sect.~\ref{sec:experiments}, we conclude the paper in Sect.~\ref{sec:conclusion}.
\section{Problem Formulation}\label{sec:preliminaries}

\paragraph*{Notations.}
Let $\RR$ be the set of real numbers. For a vector $\xx \in \RR^n$, $x_i$ refers to its $i$-th component and $\norm{\xx}$ denotes the $\ell^2$-norm. Particularly, $\zz$ and $\one$ denote respectively the vector of zeros and ones of appropriate dimension, and the comparison between vectors, e.g., $\xx \le \zz$, is component-wise. We define for $\delta > 0$, $\ball(\xx,\delta) \define \{\xx' \in \RR^n \mid \norm{\xx' - \xx} \le \delta\}$ as the $\delta$-closed ball centered at  $\xx$. We abuse the notation $\norm{\cdot}$ for an $m \times n$ matrix $M$ as $\norm{M} \define \sqrt{\sum_{i=1}^m \sum_{j=1}^n \norm{M_{i j}}^2}$. The exponential of a square matrix $M \in \RR^{n \times n}$, denoted by $\exp{M}$, is the $n \times n$ matrix given by the power series $\exp{M} \define \sum_{k=0}^\infty \frac{1}{k!} M^k$. For a set $\mathcal{X} \subseteq \RR^n$, $\boundary \mathcal{X}$, $\closure{\mathcal{X}}$ and $\mathcal{X}^\interior$ denote respectively the boundary, the closure and the interior of $\mathcal{X}$. Let $C^k$ be the space of functions on $\RR$ with continuous derivatives up to order $k$; a function $f(t, \xx)\from \RR \times \RR^n \to \RR$ is in $C^{1,2}(\RR \times \RR^n)$ if $f \in C^1$ w.r.t. $t \in R$ and $f \in C^2$ w.r.t. $\xx \in \RR^n$.

Let $(\Omega, \mathcal{F}, P)$ be a probability space, where $\Omega$ is a sample space, $\mathcal{F} \subseteq 2^\Omega$ is a $\sigma$-algebra on $\Omega$, and $P\from \mathcal{F} \to [0, 1]$ is a probability measure on the measurable space $(\Omega, \mathcal{F})$. A \emph{random variable} $X$ defined on the probability space $(\Omega, \mathcal{F}, P)$ is an $\mathcal{F}$-measurable function $X\from \Omega \to \RR^n$; its $\emph{expectation}$ (w.r.t. $P$) is denoted by $E[X]$. Every random variable $X$ induces a probability measure $\mu_X\from \mathcal{B} \to [0,1]$ on $\RR^n$, defined as $\mu_X(B) \define P(X^{-1}(B))$ for Borel sets $B$ in the Borel $\sigma$-algebra $\mathcal{B}$ on $\RR^n$. $\mu_X$ is called the \emph{distribution of $X$}; its \emph{support set} is $\supp(\mu_X) \define \overline{\bigcup_{\mu_X(B) > 0} B}$, which will also be referred to as the support of $X$.

A (continuous-time) \emph{stochastic process} is a parametrized collection of random variables $\{X_t\}_{t \in T}$ where the parameter space $T$ is interpreted as, unless explicitly notated in this paper, the halfline $[0, \infty)$. We sometimes further drop the brackets in $\{X_t\}$ when it is clear from the context. A collection $\{\mathcal{F}_t \mid t \ge 0\}$ of $\sigma$-algebras of sets in $\mathcal{F}$ is a \emph{filtration} if $\mathcal{F}_t \subseteq \mathcal{F}_{t+s}$ for $t, s \in [0, \infty)$. Intuitively, $\mathcal{F}_t$ carries the information known to an observer at time $t$. A random variable $\tau\from \Omega \to [0, \infty)$ is called a \emph{stopping time} w.r.t. some filtration $\{\mathcal{F}_t \mid t \ge 0\}$ of $\mathcal{F}$ if $\{\tau \le t\} \in \mathcal{F}_t$ for all $t \ge 0$. A stochastic process $\{X_t\}$ adapted to a filtration $\{\mathcal{F}_t \mid t \ge 0\}$ is called a \emph{supermartingale} if $E[X_t] < \infty$ for any $t \ge 0$ and $E[X_t \mid \mathcal{F}_s] \le X_s$ for all $0 \le s \le t$. That is, the conditional expected value of any future observation, given all the past observations, is no larger than the most recent observation.

\paragraph*{Stochastic Differential Dynamics.}
We consider a class of dynamical systems featuring stochastic differential dynamics governed by time-homogeneous SDEs of the form\footnote{The general time-inhomogeneous case with time-dependent $b$ and $\sigma$ can be reduced to this form (cf.~\cite[Chap.~10]{oksendal2013stochastic}).}
\begin{equation}\label{eq:sde}
\dif X_t = b(X_t) \dif t+\sigma(X_t) \dif W_t, \quad t \ge 0
\end{equation}%
where $\{X_t\}$ is an $n$-dimensional continuous-time stochastic process, $\{W_t\}$ denotes an $m$-dimensional Wiener process (standard Brownian motion), $b\from \RR^n \to \RR^n$ is a vector-valued polynomial flow field (called the \emph{drift coefficient}) modeling deterministic evolution of the system, and $\sigma\from \RR^n \to \RR^{n \times m}$ is a matrix-valued polynomial flow field (called the \emph{diffusion coefficient}) that encodes the coupling of the system to Gaussian white noise $\dif W_t$.

Suppose there exists a Lipschitz constant $D$ s.t. $\norm{b(\xx) - b(\yy)} + \norm{\sigma(\xx) - \sigma(\yy)} \le D \norm{\xx - \yy}$ holds for all $\xx, \yy \in \RR^n$. Then, given an initial state (a random variable) $X_0$, an SDE of the form~\eqref{eq:sde} has 
a unique \emph{solution} which is a stochastic process $X_t(\omega) = X(t, \omega)\from [0, \infty) \times \Omega \to \RR^n$ satisfying the stochastic integral equation (\`{a} la It\^{o}'s interpretation)
\begin{equation}\label{eq:ito-sol}
X_t = X_0 + \int_0^t b(X_s) \dif s +\int_0^t \sigma(X_s) \dif W_s.
\end{equation}%
The solution $\{X_t\}$ in Eq.~\eqref{eq:ito-sol} is also referred to as an \emph{(It\^{o}) diffusion process}, and will be denoted by $X_t^{0, X_0}$ (or simply $X_t^{X_0}$), if necessary, to indicate the initial condition $X_0$ at $t = 0$.

A great deal of information about a diffusion process can be encoded in a partial differential operator termed the \emph{infinitesimal generator}, which generalizes the Lie derivative that captures the evolution of a function along the diffusion process:

\begin{definition}[Infinitesimal generator~\cite{oksendal2013stochastic}]
	Let $\{X_t\}$ be a (time-homogeneous) diffusion process in $\mathbb{R}^n$. The \emph{infinitesimal generator $\mathcal{A}$ of $X_t$} is defined by 
	\begin{equation*}
	\mathcal{A} f(s, \xx) = \lim_{t \downarrow 0} \frac{E^{s, \xx} \left[ f(s+t, X_t) \right] - f(s, \xx)}{t}, \quad \xx \in \mathbb{R}^n.
	\end{equation*}%
	The set of functions $f\from \mathbb{R}\times\mathbb{R}^n \to \mathbb{R}$ s.t. the limit exists at $(s, \xx)$ is denoted by $\mathcal{D}_\mathcal{A}(s, \xx)$, while $\mathcal{D}_\mathcal{A}$ denotes the set of functions for which the limit exists for all $(s, \xx) \in \RR \times \mathbb{R}^n$.
\end{definition}

In subsequent sections, the readers may find applications of the operator $\mathcal{A}$ to a vector-valued function in a component-wise manner. The relation between $\mathcal{A}$ and the coefficients $b, \sigma$ in SDE~\eqref{eq:sde} is captured by the following result:

\begin{lemma}[\cite{oksendal2013stochastic}]
	Let $\{X_t\}$ be a diffusion process defined by Eq.~\eqref{eq:sde}. If $f \in C^{1, 2}(\mathbb{R}\times\mathbb{R}^n)$ with compact support, then $f \in  \mathcal{D}_\mathcal{A}$ and
	\begin{equation*}
	\mathcal{A}f(t,\xx) = \frac{\partial f}{\partial t} + \sum_{i=1}^n b_i(\xx)\frac{\partial f}{\partial x_i} + \frac{1}{2}\sum_{i,j}(\sigma \sigma^\trans)_{ij}\frac{\partial^2 f}{\partial x_i\partial x_j}.
	\end{equation*}
\end{lemma}

As a stochastic generalization of the Newton-Leibniz axiom, Dynkin's formula gives the expected value of any adequately smooth function of an It\^{o} diffusion at a stopping time:

\begin{theorem}[Dynkin's formula~\cite{dynkin1965markov}]\label{thm:dynkin}
	Let $\{X_t\}$ be a diffusion process in $\mathbb{R}^n$. Suppose $\tau$ is a stopping time with $E[\tau] < \infty$, and $f\in C^{1, 2}(\mathbb{R}\times\mathbb{R}^n)$ with compact support. Then
	\begin{equation*}
	E^{h, \xx} \left[f(\tau,X_\tau) \right] = f(h,\xx) + E^{h, \xx} \left[ \int_0^\tau \mathcal{A}f(s,X_s) \dif s \right].
	\end{equation*}
\end{theorem}

In order to specify the behavior of an It\^{o} diffusion across the domain boundary, we introduce the concept of \emph{stopped process}, which is a stochastic process that  is forced to have the same value after a prescribed (possibly random) time. 

\begin{definition}[Stopped process~\cite{gallager2013stochastic}]
	Given a stopping time $\tau$ and a stochastic process $\{X_t\}$, the \emph{stopped process} $\{X_t^\tau\}$ is defined by 
	\begin{equation*}
	X^\tau(t,\omega) \define X_{t \wedge \tau}(\omega)=
	\begin{cases}
	X(t,\omega) & \text{if } t\leq \tau(\omega),\\
	X(\tau(\omega),\omega) & \text{otherwise}.
	\end{cases}
	\end{equation*}
\end{definition}

\begin{remark}
By definition, a stopped process preserves, among others, continuity and the Markov property, and hence the aforementioned results on a stochastic process apply also to a stopped process.
\end{remark}

Now consider a stochastic system modeled by an SDE of the form~\eqref{eq:sde} that evolves ``within'' a not necessarily bounded set $\mathcal{X} \subseteq \RR^n$. Since the solution $\{X_t\}$ of Eq.~\eqref{eq:sde} may escape from $\mathcal{X}$ at any time instant $t > 0$, due to the unbounded nature of Gaussian, we define a stopped process $\tilde{X}_t \define X_{t\wedge\tau_\mathcal{X}}$ with $\tau_\mathcal{X} \define \inf\{t \mid X_t\notin \mathcal{X}\}$. $\tilde{X}_t$ hence represents the process that will stop at the boundary of $\mathcal{X}$. Denote the infinitesimal generator of the stopped process as $\tilde{\mathcal{A}}$. One plausible property here is that, for all compactly-supported $f \in C^{1,2}(\RR \times \mathbb{R}^n)$, 
\begin{equation}\label{eq:equal-generator}
\tilde{\mathcal{A}}f(t,\xx) =
\begin{cases}
\mathcal{A} f(t, \xx)& \text{for } \xx \in \mathcal{X}^\interior,\\
\frac{\partial f}{\partial t}(t, \xx)& \text{for } \xx \in \boundary \mathcal{X}.
\end{cases}
\end{equation}

\paragraph*{The $\infty$-Safety Problem.}
Given an SDE of the form~\eqref{eq:sde}, a (not necessarily bounded\footnote{In practice, if we can specify $\mathcal{X}$ based on prior knowledge when modeling a physical system, then the larger $\mathcal{X}$ we choose, the greater (bound on) failure probability we will obtain.}) domain set $\mathcal{X} \subseteq \RR^n$, an initial set $\mathcal{X}_0 \subset \mathcal {X}$, and an unsafe set $\mathcal{X}_u \subset \mathcal{X}$. We aim to bound the failure probability 
\begin{equation*}
	P\left(\exists t \in [0, \infty)\colon \tilde{X}_t \in \mathcal{X}_u \right),
\end{equation*}
for any initial state $X_0$ whose support lies within $\mathcal{X}_0$.
Accordingly, the \emph{$T$-safety problem}, with $T < \infty$, refers to the problem where one aims to bound the failure probability within the finite time horizon $[0, T]$.

\begin{remark}
	Roughly speaking, if we denote by $\phi$ the proposition ``$\tilde{X}_t$ evolves within $\mathcal{X}$'' and by $\psi$ the proposition ``$\tilde{X}_t$ evolves into $\mathcal{X}_u$'', then the above $\infty$-safety problem asks for a bound on the probability that 
	the LTL formula $\phi\, \mathcal{U} \psi$ holds.
\end{remark}

\section{Reducing $\infty$-Safety to $T$-Safety}\label{sec:reduction}

We dedicate this section to the reduction of the $\infty$-safety problem to its bounded counterpart. Observe that for any $0 \le T < \infty$,
\begin{equation*}
P(\exists t \ge 0\colon \tilde{X}_t \in \mathcal{X}_u)\le P(\exists t \in [0, T]\colon \tilde{X}_t \in \mathcal{X}_u) + P(\exists t\ge T\colon \tilde{X}_t \in \mathcal{X}_u).
\end{equation*}%
The key idea behind our approach is to first compute an exponentially decreasing bound on the \emph{tail failure probability} over $[T^*, \infty)$ (the computation of $T^* \ge 0$ will be shown later), and then for any constant $\epsilon > 0$, we can identify (out of the exponentially decreasing bound) a time instant $\tilde{T} \ge T^*$ such that $P(\exists t \ge \tilde{T}\colon \tilde{X}_t \in \mathcal{X}_u) \le \epsilon$. The overall bound on the failure probability over $[0, \infty)$ can consequently be obtained by solving the truncated $\tilde{T}$-safety problem.

\subsection{Exponentially Decreasing Bound on the Tail Failure Probability}

We first state a result that gives conditions when a linear map keeps vector inequality:
\begin{lemma}[{\cite[Chap.~4]{beckenbachinequalities}}]
For a matrix $M \in \mathbb{R}^{n \times n}$,
\begin{itemize}
\item $\forall \xx, \yy \in \mathbb{R}^n\colon \xx \le \yy \implies M \xx \le M \yy$ iff $M$ is \emph{non-negative}, i.e., $M_{ij} \ge 0$ for all $1 \le i, j \le n$.
\item The matrix $\exp{M t}$ is non-negative for all $t\ge 0$ iff $M$ is \emph{essentially non-negative}, i.e., $M_{i j} \ge 0$ for $i \neq j$.
\end{itemize}
\end{lemma}
%

The existence of an exponentially decreasing bound on the tail failure probability relies on a witness of a supermartingale of the exponential type:

\begin{theorem}\label{barrier_certificate_thm}
Suppose there exists an essentially non-negative matrix $\Lambda\in\mathbb{R}^{m\times m}$, together with an $m$-dimensional polynomial function (termed \emph{exponential stochastic barrier certificate})
$V(\xx) = \left(V_1(\xx), V_2(\xx),\ldots, V_m(\xx)\right)^\trans$, with $V_i\from \RR^n \to \RR$ for $1 \le i \le m$, satisfying\footnote{Condition~\eqref{inner_condition} is slightly stronger than the corresponding one used in~\cite{prajna2004stochastic,DBLP:journals/tac/PrajnaJP07}, yet will lead to an exponentially decreasing bound on the tail failure probability in return.}\textsuperscript{,}\footnote{Condition~\eqref{boundary_condition} is to ensure that when $\tilde{X}_t$ stops at the boundary of $\mathcal{X}$, we still have $\tilde{\mathcal{A}} V(\xx) \le -\Lambda V(\xx)$ for $\xx\in \partial \mathcal{X}$. If $\mathcal{X} = \mathbb{R}^n$, however, this condition can be omitted.}%
\begin{align}
&V(\xx)\geq \zz \quad \text{for }\xx\in \mathcal{X},\\
&\mathcal{A}V(\xx)\leq -\Lambda V(\xx) \quad \text{for }\xx\in \mathcal{X},\label{inner_condition}\\
&\Lambda V(\xx)\leq \zz \quad\text{for }\xx\in \partial \mathcal{X}.\label{boundary_condition}
\end{align}
Define a function
\begin{equation*}
F(t,\xx) \define \exp{\Lambda t}V(\xx),
\end{equation*}
then every component of $F(t,\tilde{X}_t)$ is a supermartingale.
\end{theorem}
\begin{proof}
For cases with a bounded domain $\mathcal{X}$, one can trivially extend the domain of $F(t, \xx)$ s.t. $F$ is compactly-supported, and thus Dynkin's formula in Theorem~\ref{thm:dynkin} applies immediately. For cases where $\mathcal{X}$ is unbounded, we introduce a stopping time
\begin{equation*}
\tau_\delta \define \inf\left\{t \bigm| F\left(t,\tilde{X}_t\right) \geq \ball(\zz, \delta)\right\},
\end{equation*}
and denote by $X^{(\delta)}_t \define (t \wedge \tau_\delta,\tilde{X}_{t \wedge \tau_\delta})$ the corresponding stopped process involving the timeline, and by $\mathcal{A}^{(\delta)}$ the corresponding infinitesimal generator. Then $X^{(\delta)}_t$ evolves within the $\delta$-closed ball $\ball(\zz, \delta)$ and hence boils down to the case with a bounded domain. Moreover, by Eq.~\eqref{eq:equal-generator}, we have
\begin{align*}
\mathcal{A}^{(\delta)}F\left(X_t^{(\delta)}\right) &= \mathcal{A}^{(\delta)}F\left(t \wedge \tau_\delta, \tilde{X}_{t \wedge \tau_\delta}\right)\\
& =
\begin{cases}
0\quad \text{if } \tau_\delta(\omega)\leq t,\\
\frac{\partial F}{\partial t}(t, X_t) + \exp{\Lambda t} \mathcal{A} V(X_t) \leq 0 \quad \text{if } \tau_\delta(\omega) > t \wedge \tau_\mathcal{X}(\omega) > t,\\
\frac{\partial F}{\partial t}(t, X_t) \leq 0 \quad \text{if } \tau_\delta(\omega) > t \wedge \tau_\mathcal{X}(\omega) \leq t,
\end{cases}
\end{align*}
where $\tau_\mathcal{X}$ represents the time instant when escaping from the state space $\mathcal{X}$. Note that the second and the third case hold due to the non-negativity of $\exp{\Lambda t}$ (as $\Lambda$ is essentially non-negative), which implies that $\exp{\Lambda t}$ preserves vector inequalities~\eqref{inner_condition} and~\eqref{boundary_condition}. Hence by Dynkin's formula (in a component-wise manner), for fixed $t, h \in [0, \infty)$, we have
\begin{align*}
E\left[F\left(\left(t+h\right)\wedge \tau_\delta, \tilde{X}_{(t+h)\wedge \tau_\delta}\right) \bigm| \mathcal{F}_h\right] &= E^{X_h^{(\delta)}}\left[F\left(X_{t+h}^{(\delta)}\right)\right]\\
&= F\left(X_h^{(\delta)}\right) + E^{X_h^{(\delta)}}\left[\int_0^{t}\mathcal{A}^{(\delta)}F\left(X_{s}^{(\delta)}\right) \dif s\right]\\
&\leq F\left(X_h^{(\delta)}\right)\\
&= F\left(h \wedge \tau_\delta, \tilde{X}_{h \wedge \tau_\delta}\right).
\end{align*}
Since $F(t, \xx) > \zz$, by Fatou's lemma, we have
\begin{align*}
E\left[F\left(t+h,\tilde{X}_{t+h}\right) \bigm| \mathcal{F}_h\right] &= E\left[\liminf_{\delta \to \infty} F\left((t+h) \wedge \tau_\delta, \tilde{X}_{(t+h) \wedge \tau_\delta}\right) \bigm| \mathcal{F}_h\right]\\
&\leq \liminf_{\delta \to \infty} E\left[F\left((t+h) \wedge \tau_\delta, \tilde{X}_{(t+h)\wedge\tau_\delta}\right) \bigm| \mathcal{F}_h\right]\\
&\leq \liminf_{\delta \to \infty} F\left(h \wedge \tau_\delta, \tilde{X}_{h \wedge \tau_\delta}\right)\\
&\leq F\left(h, \tilde{X}_h\right).
\end{align*}
It follows consequently that every component of $F(t,\tilde{X}_t)$ is a supermartingale.
\qed
\end{proof}

We will show in Sect.~\ref{sec:synthesis} that the synthesis of the exponential stochastic barrier certificate $V(\xx)$ (and thereby the function $F(t,\xx)$) boils down to solving a pertinent SDP optimization problem.

In order to further establish the relation between the exponential supermartingale $F(t,\tilde{X}_t)$ (and thereby $V(\xx)$) and the bound on tail failure probability, we recall Doob's maximal inequality for supermartingales, which gives a bound on the probability that a non-negative supermartingale exceeds some given value over a given time interval:
\begin{lemma}[Doob's supermartingale inequality~\cite{karatzas2014brownian}]\label{maximal_inequlity}
Let $\{X_t\}_{t > 0}$ be a right continuous non-negative supermartingale adapted to a filtration $\{\mathcal{F}_t \mid t > 0\}$. Then for any $\lambda > 0$,
\begin{equation*}
    \lambda P\left(\sup_{t \geq 0}\, X_t \geq \lambda\right) \leq E[X_0].
\end{equation*}
\end{lemma}

The following theorem claims an intermediate fact that will later reveal the exponentially decreasing bound on the tail failure probability.

\begin{theorem}\label{general_bound}
Suppose the conditions in Theorem~\ref{barrier_certificate_thm} are satisfied. Then for any $T \ge 0$ and any positive vector $\gamma \in \mathbb{R}^m$,
\begin{equation}\label{inequ:general_bound}
P\left(\sup_{t \geq T} V\left(\tilde{X}_t\right) \geq \sup_{t \geq T} \left(\exp{-\Lambda t}\gamma\right)\right) \leq E\left[V_i(X_0)\right]/\gamma_i
\end{equation}%
holds for all $i \in \{1, \ldots, m\}$.
\end{theorem}
\begin{proof}
Observe the following chain of (in-)equalities:
\begin{align*}
P\left(\sup_{t \geq T} V\left(\tilde{X}_t\right) \right.&\geq \left.\sup_{t \geq T}\left(\exp{- \Lambda t} \gamma\right)\right) \leq P\left(\exists t\geq T\colon V\left(\tilde{X}_t\right) \geq \exp{- \Lambda t} \gamma\right)\\
&\leq P\left(\exists t \geq T\colon \exp{\Lambda t} V\left(\tilde{X}_t\right) \geq \gamma\right)\TAG{non-negative $\exp{\Lambda t}$}\\
&= P\left(\sup_{t \geq T} F\left(t, \tilde{X}_t\right) \geq \gamma\right)\\
&\leq P\left(\sup_{t \geq T} F_i\left(t, \tilde{X}_t\right) \geq \gamma_i\right)\\
&\leq E\left[F_i\left(T, \tilde{X}_T\right)\right]/\gamma_i\TAG{Lemma~\ref{maximal_inequlity}}\\
&\leq E\left[V_i\left(X_0\right)\right]/\gamma_i\TAG{Theorem~\ref{barrier_certificate_thm}}
\end{align*}%
which holds for any $i\in\{1,2,\cdots,m\}$. This completes the proof.
\qed
\end{proof}

Now, we are ready to give the exponentially decreasing bound on the tail failure probability derived from Theorem~\ref{general_bound}. We start by considering the simple case where the barrier certificate $V(\xx)$ is a scalar function, i.e., with $m = 1$.

\begin{proposition}\label{dim1}
Suppose there exists a positive constant $\Lambda \in \RR$ and a scalar function $V\from \RR^n \to \RR$ satisfying Theorem~\ref{barrier_certificate_thm}. Then,
\begin{equation}\label{eq:dim1-bound-gamma}
P\left(\sup_{t \geq T} V\left(\tilde{X}_t\right) \geq \gamma\right) \leq \frac{E\left[V(X_0)\right]}{\exp{\Lambda T} \gamma}
\end{equation}
holds for any $\gamma > 0$ and $T \ge 0$. Moreover, if there exists $l > 0$ such that
\[V(\xx)\geq l \quad \text{for all } \xx\in\mathcal{X}_u,\]
then
\begin{equation}\label{eq:dim1-bound-l}
P\left(\exists t \geq T\colon \tilde{X}_t \in \mathcal{X}_u\right) \leq \frac{E[V(X_0)]}{\exp{\Lambda T} l}
\end{equation}%
holds for any $T \geq 0$.
\end{proposition}
\begin{proof}
Eq.~\eqref{eq:dim1-bound-gamma} holds since
\begin{align*}
\hspace*{-.28cm}
P\left(\sup_{t \geq T} V\left(\tilde{X}_t\right) \geq \gamma\right) &= P\left(\sup_{t \geq T} V\left(\tilde{X}_t\right) \geq \exp{- \Lambda T}\left(\exp{\Lambda T} \gamma\right)\right)\\
&\leq P\left(\sup_{t \geq T} V\left(\tilde{X}_t\right) \geq \sup_{t \geq T} \left(\exp{-\Lambda t} \left(\exp{\Lambda T}\gamma\right)\right)\right)\!\!\!\TAG{monotonicity on $t$}\\
&\leq \frac{E[V(X_0)]}{\exp{\Lambda T} \gamma}.\TAG{Theorem~\ref{general_bound}}
\end{align*}%
For Eq.~\eqref{eq:dim1-bound-l}, it is immediately obvious that
\[P\left(\exists t \geq T\colon \tilde{X}_t \in \mathcal{X}_u\right) \leq P\left(\sup_{t \geq T} V\left(\tilde{X}_t\right) \geq l\right) \leq \frac{E[V(X_0)]}{\exp{\Lambda T} l}.\]
This completes the proof.\qed
\end{proof}

Now we lift the results to the slightly more involved case with $m > 1$.

\begin{proposition}\label{high_dim}
Suppose there exists an essentially non-negative matrix $\Lambda \in \mathbb{R}^{m\times m}$ and an $m$-dimensional polynomial function
$V\from \RR^n \to \RR^m$ satisfying Theorem~\ref{barrier_certificate_thm}.
If all of the eigenvalues of $\Lambda$ have positive real parts, i.e.,
\begin{equation*}
\min_{1 \le i \le m}\left\{\Re(\lambda_i) \mid \lambda_i~\text{is an eigenvalue of}~\Lambda\right\} > 0,
\end{equation*}%
then for any positive vector $\gamma \in \mathbb{R}^m$, there exists $T^* = T^*(\gamma, M, \Lambda) \in \mathbb{R}$ such that for any $T \geq T^*$,
\begin{equation}\label{high_dim_bound_gamma}
P\left(\sup_{t \geq T} V\left(\tilde{X}_t\right) \geq \gamma\right) \leq \frac{E[V_i(X_0)]}{\left(\exp{M T} \gamma\right)_i}
\end{equation}%
holds for all $i \in \{1, \ldots, m\}$. Here, $M$ is an essentially non-negative matrix s.t. all of the eigenvalues of $\Lambda - M$ have positive real parts\footnote{Such matrix $M$ always exists, for instance, $M \define \Lambda/2$.}. Moreover, if there exists a positive vector $l \in \mathbb{R}^m$ such that
\[V(\xx) \geq l \quad \text{for all } \xx \in \mathcal{X}_u,\]
then for any $T \geq T^*$,
\begin{equation}\label{high_dim_bound}
P\left(\exists t \geq T\colon \tilde{X}_t \in \mathcal{X}_u\right) \leq \frac{E[V_i(X_0)]}{\left(\exp{M T} l\right)_i}
\end{equation}%
holds for all $i \in \{1, \ldots, m\}$.
\end{proposition}
\begin{proof}
By substituting $\gamma$ in Eq.~\eqref{inequ:general_bound} with $\exp{M T} \gamma$, we have that for all $T\geq 0$,
\begin{equation}\label{high_dim_general_bound}
\begin{aligned}
\frac{E[V_i(X_0)]}{\left(\exp{M T} \gamma\right)_i} &\geq P\left(\sup_{t \geq T} V\left(\tilde{X}_t\right) \geq \sup_{t \geq T} \left(\exp{- \Lambda t} \exp{M T} \gamma\right)\right)\\
&=P\left(\sup_{t \geq T} V\left(\tilde{X}_t\right) \geq \sup_{t \geq T} \left(\exp{- \Lambda (t-T)} \exp{- (\Lambda - M) T}\gamma\right)\right)
\end{aligned}
\end{equation}%
holds for any $\gamma \in \mathbb{R}^m$ with $\gamma > \zz$. Observe that
\begin{align*}
\infnorm{\sup_{t \geq T} \left(\exp{- \Lambda (t-T)} \exp{- (\Lambda - M) T} \gamma\right)}
&= \infnorm{\sup_{t \geq 0} \left(\exp{- \Lambda t} \exp{- (\Lambda - M) T} \gamma\right)}\\
&\leq \infnorm{\sup_{t \geq 0}\left(\exp{- \Lambda t}\right)} \infnorm{\exp{- (\Lambda - M) T} \gamma},
\end{align*}%
where $\lvert\cdot\rvert_\infty$ denotes the infinity norm. Moreover, since all of the eigenvalues of $\Lambda - M$ have positive real parts, then by the Lyapunov stability established in the theory of ODEs, we have
\[\lim_{T\to \infty}\exp{-(\Lambda - M)T} \gamma = \zz.\]
There hence exists $T^*$ s.t. for all $T \geq T^*$, 
\begin{equation}\label{high_dim_connection}
\sup_{t \geq T} \left(\exp{- \Lambda (t-T)} \exp{- (\Lambda - M) T} \gamma\right) \leq \gamma.
\end{equation}%
By Combining Eq.~\eqref{high_dim_connection} and Eq.~\eqref{high_dim_general_bound}, we obtain Eq.~\eqref{high_dim_bound_gamma}. For Eq.~\eqref{high_dim_bound}, it follows immediately that 
\[P\left(\exists t \geq T\colon \tilde{X}_t \in \mathcal{X}_u\right) \leq P\left(\sup_{t \geq T} V\left(\tilde{X}_t\right) \geq l\right) \leq \frac{E[V_i(X_0)]}{\left(\exp{M T} l\right)_i}.\]
This completes the proof.
\qed	
\end{proof}

\begin{remark}
Proposition~\ref{high_dim} argues the existence of $T^*$ that suffices to ``split off'' the tail failure probability. From a computational perspective, this is algorithmically tractable as the matrix exponential involved in Eq.~\eqref{high_dim_connection} is symbolically computable (cf., e.g.,~\cite{moler2003nineteen}).
\end{remark}

The following theorem states the main result of this section, that is, for any given constant $\epsilon$, there exists $\tilde{T} \ge 0$ such that the truncated $\tilde{T}$-tail failure probability is bounded by $\epsilon$:

\begin{theorem}\label{vector_barrier_conclusion}
Suppose the conditions in Proposition~\ref{dim1} and~\ref{high_dim} are satisfied. If there exists $\alpha > 0$, s.t. $\forall \xx \in \mathcal{X}_0\colon V_i(\xx) \leq \alpha$ holds for some $i \in \{1, \ldots, m\}$.
Then for any $\epsilon > 0$, there exists $\tilde{T} \geq 0$ such that 
\begin{equation*}
P\left(\exists t \geq \tilde{T}\colon \tilde{X}_t \in \mathcal{X}_u\right)\leq \epsilon.
\end{equation*}	
\end{theorem}
\begin{proof}
Observe that for Eq.~\eqref{high_dim_bound} in Proposition~\ref{high_dim}, the assumption $\forall \xx \in \mathcal{X}_0\colon V_i(\xx) \leq \alpha$ guarantees an upper bound on the numerator $E[V_i(X_0)]$, while the essential non-negativity of $M$ (with all its eigenvalues having positive real parts) ensures that the denominator $(\exp{M T} l)_i \to +\infty$ as $T \to \infty$. An analogous argument applies to Eq.~\eqref{eq:dim1-bound-l} in Proposition~\ref{dim1}. The claim in this theorem then follows immediately. 
\qed	
\end{proof}
%
%

\subsection{Bounding the Failure Probability over $[0, T]$}
The reduced $T$-safety problem can be solved by existing methods tailored for bounded verification of SDEs, e.g.,~\cite{steinhardt2012finite,DBLP:conf/ccta/SantoyoDC19}. In what follows, we propose an alternative method leveraging \underline{time-dependent} polynomial stochastic barrier certificates. Our method requires constraints (on the barrier certificates) of simpler form compared to~\cite{steinhardt2012finite}; meanwhile, it yields strictly more expressive form of barrier certificates, against the approach on unbounded verification as in~\cite{prajna2004stochastic,DBLP:journals/tac/PrajnaJP07}, thus leading to theoretically non-looser (usually tighter) failure bound. A detailed argument will be given at the end of this section.

The following theorem states a sufficient condition, i.e., a collection of constraints on the time-dependent polynomial stochastic barrier certificates $H(t, \xx)$, under which the failure probability of a stochastic system over a finite time horizon can be explicitly bounded from above.

\begin{theorem}\label{finite_time_condition}
	Suppose there exists a constant $\eta>0$ and a polynomial function (termed \emph{time-dependent stochastic barrier certificate}) $H(t, \xx)\from \mathbb{R} \times \mathbb{R}^n \to \mathbb{R}$, satisfying\footnote{Condition \eqref{finite_time_boundary_condition} is to ensure that when $\tilde{X}_t$ stops at the boundary of $\mathcal{X}$, we still have $\tilde{A} H(t, \xx) \leq 0$ for $\xx \in \boundary \mathcal{X}$. If $\mathcal{X} = \RR^n$, however, this condition can be dropped.}
	\begin{align}
	&H(t, \xx) \geq 0 \quad \text{for } (t, \xx) \in [0, T] \times \mathcal{X},\\
	&\mathcal{A} H(t, \xx) \leq 0 \quad \text{for } (t,\xx) \in [0,T] \times \left(\mathcal{X} \setminus \mathcal{X}_u\right),\label{finite_time_supermartingale_condition}\\
	&\frac{\partial H}{\partial t} \leq 0 \quad \text{for } (t, \xx) \in [0, T] \times \partial\mathcal{X},\label{finite_time_boundary_condition}\\
	&H(t, \xx) \geq \eta \quad \text{for } (t, \xx) \in [0, T] \times \mathcal{X}_u.
	\end{align}
	Then,
	\begin{equation}
	P\left(\exists t \in [0, T]\colon \tilde{X}_t \in \mathcal{X}_u\right) \leq \frac{E[H(0, X_0)]}{\eta}.
	\end{equation}
\end{theorem}
\begin{proof}
	Assume in the following that the system evolves within a bounded domain $\mathcal{X}$\footnote{For cases with an unbounded $\mathcal{X}$, the same proof technique of introducing a $\delta$-closed ball as in the proof of Theorem~\ref{barrier_certificate_thm} applies.}. Define a stopping time 
	\[
	\tau_u \define \inf\left\{t \bigm| \tilde{X}_t \notin \mathcal{X} \setminus \mathcal{X}_u\right\},
	\]
	and denote by $X^{(u)}_t \define (t \wedge \tau_u \wedge T, \tilde{X}_{t \wedge \tau_u \wedge T})$ the corresponding stopped process, and by $\mathcal{A}^{(u)}$ the corresponding infinitesimal generator. By Eq.~\eqref{eq:equal-generator}, we have
	\begin{align*}
	\mathcal{A}^{(u)} H\left(X_t^{(u)}\right) &= \mathcal{A}^{(u)} H\left(t \wedge \tau_u \wedge T, \tilde{X}_{t \wedge \tau_u \wedge T}\right)\\
	&=
	\begin{cases}
	0\quad \text{if } t \geq T \vee t \geq \tau_u(\omega),\\
	\mathcal{A} H(t, X_t) \leq 0 \quad \text{if } t < \min\{T, \tau_u(\omega), \tau_\mathcal{X}(\omega)\},\\
	\frac{\partial H}{\partial t}(t, X_t) \leq 0 \quad \text{if } t < \min\{T, \tau_u(\omega)\} \wedge t \geq \tau_\mathcal{X}(\omega).
	\end{cases}
	\end{align*}
	By Dynkin's formula, for fixed $t, h \in [0, T]$, we have
	\begin{align*}
	E\left[H\left(X_{t+h}^{(u)}\right) \bigm| \mathcal{F}_h\right] &= E^{X_{h}^{(u)}} \left[H\left(X_{t+h}^{(u)}\right)\right]\\
	&= E\left[H\left(X_{h}^{(u)}\right)\right] + E^{X_{h}^{(u)}}\left[\int_0^t \mathcal{A}^{(u)} H\left(X_s^{(u)}\right) \dif s\right]\\
	&\leq E\left[H\left(X_{h}^{(u)}\right)\right].
	\end{align*}
	Thus $H(X_{t}^{(u)})$ is a non-negative supermartingale. Then by Doob's maximal inequality in Lemma~\ref{maximal_inequlity}, we have
	\begin{align*}
	P\left(\exists t \in [0, T]\colon \tilde{X}_t \in \mathcal{X}_u\right) &= P\left(\exists t \ge 0\colon \tilde{X}_{t \wedge \tau_u \wedge T} \in \mathcal{X}_u\right)\\
	&\leq P\left(\exists t \geq 0\colon H\left(X_{t}^{(u)}\right) \geq \eta\right)\\
	&\leq \frac{E[H(0, X_0)]}{\eta}.
	\end{align*}
	This completes the proof.
	\qed
\end{proof}

The following fact is then immediately obvious:

\begin{corollary}\label{finite_time_conclusion}
Suppose the conditions in Theorem~\ref{finite_time_condition} hold, and there exists $\beta > 0$, s.t. $H(0, \xx) \leq \beta$ for $\xx \in \mathcal{X}_0$. Then,
\begin{equation*}
    P\left(\exists t \in [0, T]\colon \tilde{X}_t \in \mathcal{X}_u\right) \leq \frac{\beta}{\eta}.
\end{equation*}
\end{corollary}
\begin{proof}
This is a direct consequence of Theorem~\ref{finite_time_condition}.
	\qed	
\end{proof}

\paragraph*{Remarks on Potentially Tighter Bound.}
There exists already in the literature a barrier certificate-based method proposed in~\cite{prajna2004stochastic,DBLP:journals/tac/PrajnaJP07} that can deal with the $\infty$-safety problem. It is worth highlighting, however, that our bound on the overall failure probability derived from Proposition~\ref{dim1},~\ref{high_dim} and Theorem~\ref{finite_time_condition} (with appropriate $\tilde{T}$ chosen) is at least as tight as (and usually tighter than, as can be seen later in the experiments) that in~\cite{prajna2004stochastic,DBLP:journals/tac/PrajnaJP07}. The reasons are twofold: (1) the reduction to a finite-time horizon $\tilde{T}$-safety problem substantially ``trims off'' verification efforts pertaining to $t > \tilde{T}$; (2) our method for the reduced $\tilde{T}$-safety problem admits time-dependent barrier certificates, which are strictly more expressive than those time-independent ones exploited in~\cite{prajna2004stochastic,DBLP:journals/tac/PrajnaJP07}, in the sense that any feasible solution thereof shall also be a feasible solution satisfying Theorem~\ref{finite_time_condition}.

\begin{remark}
Roughly speaking, by setting the diffusion coefficients $\sigma$ in SDEs to zero, our method applies trivially to ODE dynamics with either a known or an unknown probability distribution over the initial set of states. For the former, we can even obtain a tighter bound on the failure probability, since in this case we do not need to compute a bound on the barrier certificate over all possible initial distributions.
\end{remark}
\section{Synthesizing Stochastic Barrier Certificates Using SDP}\label{sec:synthesis}

In this section, we encode the synthesis of the aforementioned exponential and time-dependent stochastic barrier certificates into semidefinite programming~\cite{wolkowicz2012handbook} optimizations, and thus a solution thereof yields an upper bound on the failure probability over the infinite-time horizon. Specifically, an SDP problem is formulated, for each of the two barrier certificates, to encode the constraints for ``being an exponential/time-dependent stochastic barrier certificate'', while in the meantime optimizing the tightness of the failure probability bound.
%

It is worth noting that SDP is a generalization of the standard linear programming in which the element-wise non-negativity constraints are replaced by a generalized inequality w.r.t. the cone of positive semidefinite matrices. The generalization preserves \emph{convexity}, leading to the fact that SDP admits polynomial-time algorithms, say the well-known \emph{interior-point methods}, that can efficiently solve the synthesis problem, albeit numerically. We remark that the numerical computation employed in off-the-shelf SDP solvers and the use of interior-point algorithms may potentially lead to erroneous results and thereby unsoundness in the verification/synthesis results. There have been numerous attempts to validate the results from the solver through a-posteriori numerical verification of the solution. For more details, we refer the readers to~\cite{DBLP:journals/fmsd/RouxVS18} and the references therein.

\paragraph*{Exponential Stochastic Barrier Certificate $V(\xx)$.}
To encode the synthesis problem into an SDP optimization, we first fix the dimension $m$ together with $\Lambda$ satisfying Proposition~\ref{dim1} or~\ref{high_dim} (depending on $m$), and then assume a polynomial template $V^a(\xx)$ of certain degree $k$ with unknown parameters $a$, as the barrier certificate to be discovered. It then suffices to solve the following SDP problem\footnote{SDP problems in this paper refer to those that can be readily translated into the standard form of SDP, through, e.g., Stengle's Positivstellensatz~\cite{stengle1974nullstellensatz} and sum-of-squares decomposition~\cite{Parillo/2003/Semidefinite}.}:
\begin{align}
\underset{a, \alpha}{\minimize}\quad &\alpha\label{eq:sdp-exponential-obj}\\
\subj\quad &V^a(\xx)\geq \zz \for \xx \in \mathcal{X}\label{eq:sdp-exponential-cons1}\\
&\mathcal{A} V^a(\xx)\leq -\Lambda V^a(\xx)\for \xx \in \mathcal{X}\label{eq:sdp-exponential-cons2}\\
&\Lambda V^a(\xx)\leq \zz \for \xx \in \partial \mathcal{X}\label{eq:sdp-exponential-cons3}\\
&V^a(\xx)\geq \one \for \xx \in \mathcal{X}_u\label{eq:sdp-exponential-cons4}\\
&V^a(\xx)\leq \alpha\one \for \xx \in \mathcal{X}_0\label{eq:sdp-exponential-cons5}
\end{align}%
Here, the constraints~\eqref{eq:sdp-exponential-cons1}--\eqref{eq:sdp-exponential-cons3} encode the definition of an exponential stochastic barrier certificate (cf. Theorem~\ref{barrier_certificate_thm}), while constraint~\eqref{eq:sdp-exponential-cons4} (resp.,~\eqref{eq:sdp-exponential-cons5}) corresponds to the lower (resp., upper) bound of $V(\xx)$ as in Proposition~\ref{dim1} and~\ref{high_dim} (resp., Theorem~\ref{vector_barrier_conclusion})\footnote{The lower bound $l$ of $V(\xx)$ in Proposition~\ref{dim1} and~\ref{high_dim} is normalized to a vector with all its components no less than 1, based on the observation that, for any $c > 0$, $ V^a(\xx)$ is a feasible solution implies $c V^a(\xx)$ is also a feasible solution.}. Hence, minimizing the upper bound $\alpha$ of (each component of) $V^a(\xx)$ gives a tight exponentially decreasing bound on the tail failure probability, as claimed in Proposition~\ref{dim1} and~\ref{high_dim}.
%

\begin{remark}
If $\Lambda$ is chosen as a non-negative matrix, the combination of condition~\eqref{eq:sdp-exponential-cons1} and~\eqref{eq:sdp-exponential-cons3} will force $V^a(\xx) = \zz$ for $\xx \in \partial\mathcal{X}$, whereof the strict equality may be violated due to numerical computations in SDP. In practice, however, this issue can be well addressed by looking for a barrier certificate of the form $g(\xx)V(\xx)$, where $g(\xx)$ satisfies $\partial \mathcal{X} \subseteq \{\xx \mid g(\xx) = 0\}$, namely, an overapproximation of the boundary of $\mathcal{X}$.
\end{remark}

\begin{remark}
The choice of $m$ is arbitrary, while the choices of $\Lambda$ and $k$ can be heuristic: If $\Lambda_1$ admits no feasible solution, neither will $\Lambda_2 \ge \Lambda_1$ (point-wise, with all the rest parameters fixed); similarly, if $k_1$ admits no feasible solution, neither will $k_2 \le k_1$ (with all the rest parameters fixed). Therefore, one may decrease $\Lambda$ (say, by a half) or increase $k$ (say, by one) whenever a valid barrier certificate was not found.
\end{remark}

\paragraph*{Time-Dependent Stochastic Barrier Certificate $H(t, \xx)$.}
Given the results established in Sect.~\ref{sec:reduction}, the corresponding synthesis problem can be analogously encoded as the following SDP problem:
\begin{align}
\underset{b, \beta}{\minimize}\quad &\beta\label{eq:sdp-dependent-obj}\\
\subj\quad &H^b(t, \xx) \geq 0 \for (t, \xx) \in [0, T] \times \mathcal{X}\label{eq:sdp-dependent-cons1}\\
&\mathcal{A} H^b(t, \xx) \leq 0 \for (t,\xx) \in [0,T] \times \left(\mathcal{X} \setminus \mathcal{X}_u\right)\label{eq:sdp-dependent-cons2}\\
&\frac{\partial H^b}{\partial t} \leq 0 \for (t, \xx) \in [0, T] \times \partial\mathcal{X}\label{eq:sdp-dependent-cons3}\\
&H^b(t,\xx)\geq 1 \for (t,\xx)\in [0,T]\times \mathcal{X}_u\label{eq:sdp-dependent-cons4}\\
&H^b(0,\xx)\leq \beta \for \xx\in \mathcal{X}_0\label{eq:sdp-dependent-cons5}
\end{align}%
Similarly, the constraints~\eqref{eq:sdp-dependent-cons1}--\eqref{eq:sdp-dependent-cons4} encode the definition of a time-dependent stochastic barrier certificate (cf. Theorem~\ref{finite_time_condition}), while constraint~\eqref{eq:sdp-dependent-cons5} corresponds to the upper bound of $H(t, \xx)$ as in Corollary~\ref{finite_time_conclusion} (with $\eta$ being normalized to 1, as in constraint~\eqref{eq:sdp-dependent-cons4}). Consequently, minimizing the upper bound $\beta$ of $H^b(t, \xx)$ produces a tight bound on the failure probability over the reduced finite-time horizon, as stated in Corollary~\ref{finite_time_conclusion}.

\begin{remark}
The state-of-the-art interior-point methods solve an SDP problem up to an error $\varepsilon$ in time that is polynomial in the program description size (number of variables) and $\log(1/\varepsilon)$. The former is exponential in the degree of $V^a$ and $H^b$, as it corresponds to the number of monomials in the template polynomials.
\end{remark}
\section{Implementation and Experimental Results}\label{sec:experiments}

To further demonstrate the practical performance of our approach, we have carried out a prototypical implementation in~\textsc{Matlab} R2019b, with the toolbox~\textsc{Yalmip}~\cite{Lofberg2004} and~\textsc{Mosek}~\cite{DBLP:journals/mp/AndersenRT03} equipped for formulating and solving the underlying SDP problems. Given an $\infty$-safety problem as input, our implementation works toward an upper bound on the failure probability over the infinite time horizon, leveraging the reduction to a $T$-safety problem based on a computed exponentially decreasing bound on the tail failure probability. A collection of benchmark examples from the literature has been evaluated on a 1.8GHz Intel Core-i7 processor with 8GB RAM running 64-bit Windows 10. Each of the examples has been successfully tackled within 30 seconds. In what follows, we demonstrate the applicability of our techniques to SDEs featuring different dimensionalities and nonlinear dynamics, and show particularly that our approach usually produces tighter bounds compared to existing methods.
 
\begin{example}[Population growth \cite{panik2017stochastic}] Consider the stochastic system
\begin{equation*}
\dif X_t = b\left(X_t\right)\dif t + \sigma\left(X_t\right) \dif W_t,
\end{equation*}
which is a stochastic model of population dynamics subject to random fluctuations that, possibly, can be attributed to extraneous or chance factors such as the weather, location, and the general environment. Suppose that the state space is restricted within $\mathcal{X} = \{\xx \mid \xx \ge 0\}$ with $b(X_t) = -X_t$ and $\sigma(X_t) = \sqrt{2}/{2}X_t$. We instantiate the $\infty$-safety problem as $\mathcal{X}_0 = \{\xx \mid \xx = 1\}$ and $\mathcal{X}_u = \{\xx \mid \xx \ge 2\}$, namely, we expect that the population does not diverge beyond 2.

Let $\Lambda = 1$ (with $m = 1$) and set the polynomial template degree of the exponential stochastic barrier certificate $V^a(\xx)$ to $4$, the SDP solver gives \begin{align*}
V^a(\xx) = &~0.000001474596322 - 0.000044643990040 \xx\\
&+ 0.125023372121222 \xx^2 + 0.000000001430428 \xx^3,
\end{align*}%
which satisfies
\[V^a(\xx) \geq 1 ~~\text{for}~ \xx \in \mathcal{X}_u \quad \text{and} \quad V^a(\xx) \leq 0.12498 ~~\text{for}~ \xx \in \mathcal{X}_0.\]
Thus by Proposition~\ref{dim1}, we obtain the exponentially decreasing bound
\begin{equation*}
P\left(\exists t \geq T\colon \tilde{X}_t \in \mathcal{X}_u\right) \leq \frac{0.12498}{\exp{T}} \quad \text{for all } T>0.
\end{equation*}%
The user then may choose any $T > 0$ and solve the reduced $T$-safety problem. As depicted in the left of Fig.~\ref{fig:tightness}, different choices lead to different bounds on the failure probability. Nevertheless, one may surely select an appropriate $T$ that yields a way tighter overall bound on the failure probability than that produced by the method in~\cite{prajna2004stochastic,DBLP:journals/tac/PrajnaJP07}.
\end{example}

\begin{figure}[t]
\centering
\begin{minipage}[b]{.45\linewidth}
\begin{tikzpicture}
\draw (0, 0) node[inner sep=0] {\includegraphics[width=1\linewidth]{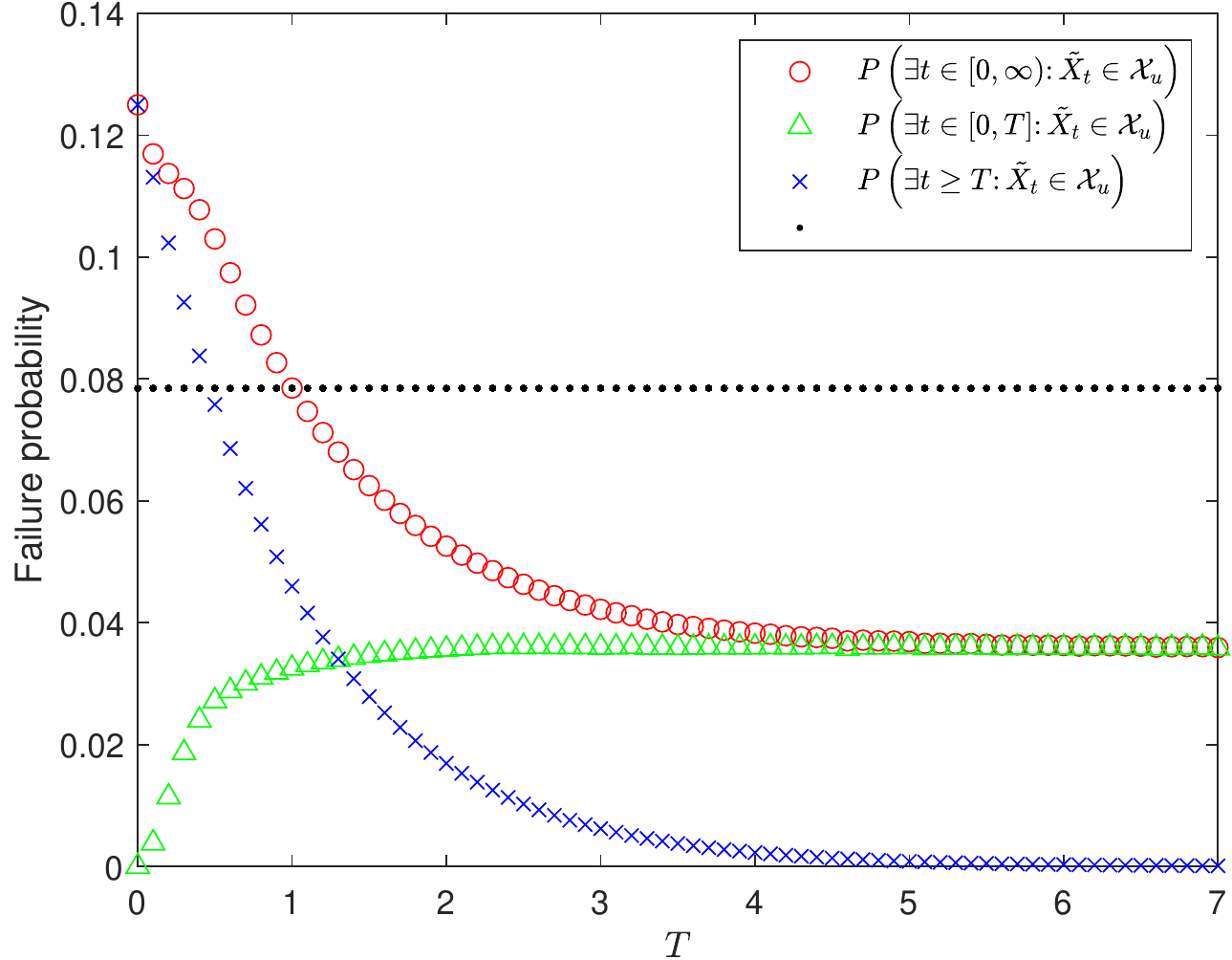}};
\draw (1.69, 1.11) node {\tiny bound by~\cite{prajna2004stochastic,DBLP:journals/tac/PrajnaJP07}};
\end{tikzpicture}
\end{minipage}
$\quad$
\begin{minipage}[b]{.45\linewidth}
\begin{tikzpicture}
\draw (0, 0) node[inner sep=0] (image) {\includegraphics[width=1\linewidth]{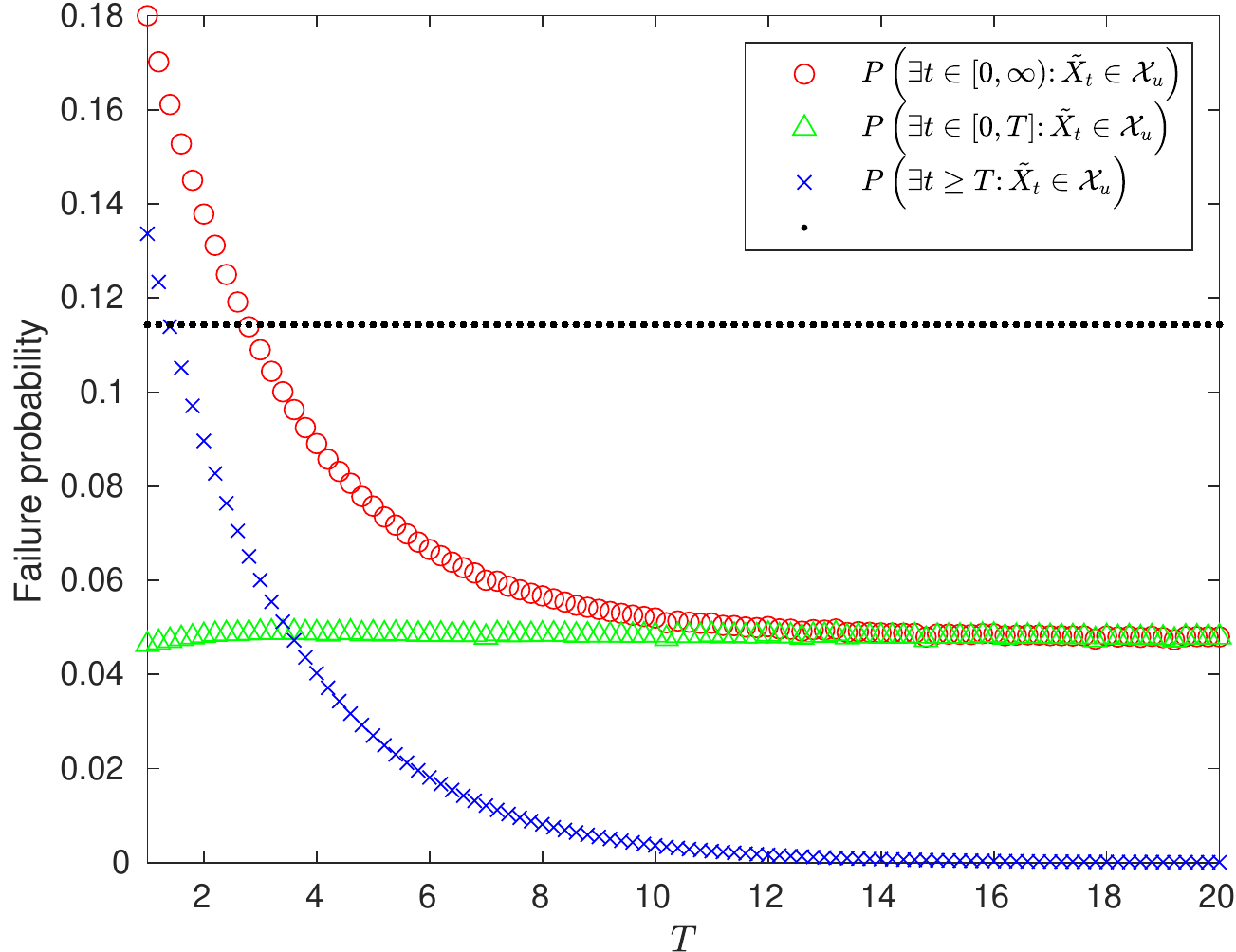}};
\draw (1.68, 1.11) node {\tiny bound by~\cite{prajna2004stochastic,DBLP:journals/tac/PrajnaJP07}};
\end{tikzpicture}
\end{minipage}
\caption{Different choices of $T$ lead to different bounds on the failure probability (with the time-dependent stochastic barrier certificates of degree $4$). Note that `$\textcolor{red}{\circ}$' $=$ `$\textcolor{blue}{\times}$' + `\raisebox{-.15mm}{$\textcolor{green}{\triangle}$}' and `$\vcenter{\hbox{\tiny$\bullet$}}$' depicts the overall bound on the failure probability produced by the method in~\cite{prajna2004stochastic,DBLP:journals/tac/PrajnaJP07}.}
\label{fig:tightness}
\end{figure}

\begin{example}[Harmonic oscillator~\cite{hafstein2018lyapunov}]
Consider a two-dimensional harmonic oscillator with noisy damping:
\begin{equation*}
\dif X_t = 
\begin{pmatrix}
0 & \omega\\
-\omega & -k
\end{pmatrix}
X_t\dif t +
\begin{pmatrix}
0 & 0\\
0 & -\sigma
\end{pmatrix}
X_t \dif W_t,\\
\end{equation*}%
with constants $\omega = 1, k = 7$ and $\sigma = 2$.
We instantiate the $\infty$-safety problem as $\mathcal{X} = \mathbb{R}^n$, $\mathcal{X}_0 = \{(x_1, x_2) \mid -1.2\leq x_1\leq 0.8, -0.6\leq x_2\leq 0.4\}$ and $\mathcal{X}_u=\{(x_1,x_2) \mid \abs{x_1} \ge 2\}$.

Let
$\Lambda =
\begin{pmatrix}
0.45 & 0.1\\
0.1 & 0.45
\end{pmatrix}
$%
and set the polynomial template degree of the exponential stochastic barrier certificate $V^a(\xx)$ to $4$, the SDP solver produces a two-dimensional $V^a(\xx)$ (abbreviated for clear presentation) satisfying
\[
V^a(\xx) \leq
\begin{pmatrix}
0.19946\\
0.19946
\end{pmatrix}
~~\text{for}~\xx \in \mathcal{X}_0 \quad \text{and} \quad V^a(\xx) \geq l =
\begin{pmatrix}
1.000237\\
1.000236
\end{pmatrix}
~~\text{for}~\xx \in \mathcal{X}_u.
\]
According to the proof of Proposition~\ref{high_dim}, we set
$M =
\begin{pmatrix}
0.3 & 0.1\\
0.1 & 0.3
\end{pmatrix}
$%
and aim to find $T^* \ge 0$ such that for all $T \ge T^*$, 
\begin{equation}\label{eq:oscillator-T}
\sup_{t \geq 0}\left(\exp{-\Lambda t}\exp{-(\Lambda-M)T}
\begin{pmatrix}
1.000237\\
1.000236
\end{pmatrix}
\right) \leq
\begin{pmatrix}
1.000237\\
1.000236
\end{pmatrix}.
\end{equation}%
Symbolic computation on the matrix exponential gives
\begin{small}
\begin{align*}
\sup_{t \geq 0}\left(
\!
\exp{-\Lambda t}\exp{-(\Lambda-M)T}
\!
\begin{pmatrix}
1.000237\\
1.000236
\end{pmatrix}
\!\!
\right)
\!&=
\sup_{t\geq0} 
\begin{pmatrix}
\exp{-0.15 T} (1.0002365 \exp{-0.55 t} + 0.0000005 \exp{-0.35 t})\\
\exp{-0.15 T} (1.0002365 \exp{-0.55 t} - 0.0000005 \exp{-0.35 t})
\end{pmatrix}\\
&\leq
\begin{pmatrix}
1.0002365\exp{-0.15T}\\
1.0002365\exp{-0.15T}
\end{pmatrix}.
\end{align*}%
\end{small}%
Therefore, $T^* = 1$ satisfies condition~\eqref{eq:oscillator-T}. Further by Corollary~\ref{high_dim}, for any $T \geq T^* = 1$, we have
\[
P\left(\exists t \geq T\colon \tilde{X}_t \in \mathcal{X}_u\right) \leq \frac{E[V_1(X_0)]}{(\exp{M T}l)_1} \leq \frac{0.19946}{0.0000005 \exp{0.2 T} + 1.00024 \exp{0.4 T}}.
\]
Analogously, a comparison with existing methods concerning the tightness of the synthesized failure probability bound (under different choices of $T$) is shown in the right of Fig.~\ref{fig:tightness}.
\end{example}

\begin{example}[Nonlinear drift~\cite{prajna2004stochastic}] We consider in this example a stochastic system involving nonlinear dynamics in its drift coefficient:
\begin{align*}
\dif x_1(t) &= x_2(t) \dif t\\
\dif x_2(t) &= - x_1(t) - x_2(t) - 0.5 x_1^3(t) \dif t + 0.1 \dif W_t.
\end{align*}%
As in~\cite{prajna2004stochastic}, let $\mathcal{X} = \{(x_1,x_2) \mid \abs{x_1} \leq 3, \abs{x_2} \leq 3, x_1^2+x_2^2 \geq 0.5^2\}$, $\mathcal{X}_0 = \{(x_1, x_2) \mid (x_1+2)^2+x_2^2 \leq 0.1^2\}$ and $\mathcal{X}_u = \{(x_1,x_2) \in \mathcal{X} \mid x_2 \geq 2.25\}$. With $\Lambda=1.5$ ($m=1$), we obtain an exponential stochastic barrier certificate $V^a(\xx)$ of degree $8$ satisfying
\[V^a(\xx) \leq 4.00014~~\text{for}~\xx \in \mathcal{X}_0 \quad \text{and} \quad V^a(\xx) \geq 1.05248~~\text{for}~\xx \in \mathcal{X}_u.\]
Thus by Corollary~\ref{dim1}, we have for any $T \geq 0$,
\[P\left(\exists t \geq T\colon \tilde{X}_t \in \mathcal{X}_u\right) \leq \frac{3.80070}{\exp{1.5 T}}.\]
Setting, for instance, $T = 6$, we have
\begin{equation*}
P\left(\exists t \geq 0\colon \tilde{X}_t \in \mathcal{X}_u\right) \leq P\left(\exists t \in [0, 6]\colon \tilde{X}_t \in \mathcal{X}_u\right) + \frac{3.80070}{\exp{9}}.
\end{equation*}%
For the reduced $T$-safety problem with $T = 6$, a time-dependent stochastic barrier certificate of degree $8$ is synthesized, thereby yielding $P\left(\exists t \in [0, 6]\colon \tilde{X}_t \in \mathcal{X}_u\right) \leq 0.196124$, thus together we get 
\begin{equation*}
P\left(\exists t \geq 0\colon \tilde{X}_t \in \mathcal{X}_u\right) \leq 0.196593,
\end{equation*}%
which is tighter than $0.265388$ produced (on the same machine) by the method in~\cite{prajna2004stochastic} under the same template degree.
\end{example}
\section{Conclusion}\label{sec:conclusion}

We proposed a constructive method, based on the synthesis of stochastic barrier certificates, for computing an exponentially decreasing upper bound, if existent, on the tail probability that an SDE system violates a given safety specification. We showed that such an upper bound facilitates a reduction of the verification problem over an unbounded temporal horizon to that over a bounded one. Preliminary experimental results on a set of interesting examples from the literature demonstrated the effectiveness of the reduction and that our method often produces tighter bounds on the failure probability.

For future work, we plan to investigate a possible convergence result in the sense that the derived failure probability bound may converge to the exact one as increasing the degree of the barrier certificates. Extending our technique to tackle SDEs with control inputs will also be of interest. Moreover, checking whether a given parametric (polynomial) formula keeps probabilistic invariance plays a central in the verification of SDEs. Several kinds of sufficient conditions on probabilistic barrier certificates were proposed, including the ones given in this paper. It consequently deserves to investigate a necessary and sufficient condition for checking the probabilistic invariance of a given template, like for ODEs in~\cite{LZZ11}. Apart from that, we are interested in carrying our results to the verification of probabilistic programs without conditioning, which can be viewed as discrete-time stochastic dynamics.


\bibliographystyle{abbrv}
\bibliography{references}

\end{document}